\documentclass[11pt]{amsart} % use larger type; default would be 10pt

\usepackage[utf8]{inputenc} % set input encoding (not needed with XeLaTeX)
\usepackage{amsmath}
\usepackage{amssymb}
\usepackage{mathtools}
\usepackage{mdwlist}
\usepackage{array}
\usepackage{url}

\def\scr{\mathcal}

\def\Conv{\operatorname{Conv}}

\def\ditemfirst#1#2{\begin{enumerate}\item \label{#1} #2\suspend{enumerate}}

\newtheorem{lem}{Lemma}

\newtheorem{thm}{Theorem}
\newtheorem{cor}{Corollary}
\usepackage{graphicx} % support the \includegraphics command and options

\usepackage{array} % for better arrays (eg matrices) in maths
\def\<{\langle}
\def\>{\rangle}

\def\R{\mathbb R}

\def\Power{\mathcal P}

\title{Proper Scoring Rules and Domination}
\author{Alexander R. Pruss}
\begin{document}
%\raggedright
\sloppy
\begin{abstract}
I generalize a theorem of Predd, et al.~(2009) on domination and strictly proper scoring rules to the case of non-additive scoring rules.
\end{abstract}

\maketitle
Let $\Omega$ be a finite space. Let $\scr C$ be the set of all functions from the power set $\Power\Omega$ to the reals: these we call
credence functions. Let $\scr P$ be
the subset of $\scr C$ which consists of the functions satisfying the axioms of probability. A scoring or inaccuracy rule is a function
$s$ from $\scr C$ to $[0,\infty]^\Omega$. Thus, $s(c)$ for a $c\in\scr C$ is a function from $\Omega$ to $[0,\infty]$, and its value
at $\omega\in\Omega$ represents how inaccurate having credence $c$ is if the true state is $\omega$.

Given a probability $p \in \scr P$, let $E_p$ be the expected value with respect to $p$. Then a scoring rule $s$ is said to be
proper provided that for all $p\in \scr P$ and $c\in\scr C$, $E_p s(p) \le E_p s(c)$, where
$$
    E_p f = \sum_{\omega \in \Omega, p(\{\omega\})\ne 0} p(\{\omega\}) f(\omega)
$$
is the mathematical expectation with a tweak to ensure well-definition in case $f$ is infinite at some points with probability zero.
If the inequality $E_p s(p) \le E_p s(c)$ is always strict, the rule is said to be strictly proper.

Enumerate the elements of $\Omega$ as $\omega_1,\dots,\omega_n$.
Let $T$ be the set of vectors $v=(v_1,\dots,v_n)$ in $\R^n$ with $v_i\ge 0$ and $\sum_i v_i = 1$. For $v\in T$, let $p_v$ be the probability
function such that $p_v(\{\omega_i\}) = v_i$ for all $i$. Say that a scoring rule is continuous on probabilities if the function
$\hat s(v) = s(p_v)$ is a continuous function from $T$ to $[0,\infty]^\Omega$, where the latter space is equipped with the product topology.

Say that a member of $[0,\infty]^\Omega$ is finite if it never takes on the value $\infty$; otherwise, say it's infinite.
Let $S_s = \{ s(p) : p\in \scr P \}$ be the set of all scores of probabilities, and let $F_s = S_s \cap [0,\infty)^\Omega$
be the set of finite scores.

\begin{thm}\label{th:domination} Suppose $s$ is strictly proper and that $S_s$ is the closure of $F_s$. Then for any
    $c\in \scr P\backslash\scr C$, there is a probability $p$ such that $s(p) < s(c)$ everywhere on $\Omega$.
\end{thm}

In other words, given the stated conditions, every score of a non-probability is dominated by the score of a probability (lower scores are better).

\begin{cor}\label{cor:domination} Suppose $s$ is strictly proper and continuous on the probabilities. Then for any
    $c\in \scr P\backslash\scr C$, there is a probability $p$ such that $s(p) < s(c)$ everywhere on $\Omega$.
\end{cor}

In the special case of additive scoring rules, this was proved by Predd, et al. (2009).
This result in the case of strict propriety was announced by Richard Pettigrew (2022), but his proof appears deficient. Michael Nielsen
has discovered a proof of this domination result using different methods during approximately the same time period
as I did.

\begin{proof}[Proof of Corollary~\ref{cor:domination}]
    The set $S_s$ is the image of the compact set $T$ under the continuous mapping $v\mapsto s(p_v)$ so it is compact.

    It remains to show that it is equal to $\overline {F_s}$. Note that $[0,\infty]^\Omega$ is metrizable (e.g., one can use the
    metric $d(a,b)=\sum_i |\arctan a_i-\arctan b_i|$), so we can work in terms of sequences. Suppose $(s_n)$ is
    a convergent sequence in $F_s$. Then $s_n = \hat s(v_n)$ for a sequence $(v_n)$ in $T$. Passing to a subsequence,
    by compactness of $T$, we may assume $v_n$ converges to $v\in T$ is compact. Thus, $s_n$ converges to $\hat s(v)$
    by continuity of $\hat s$, and hence the limit of $s_n$ is in $S_s$. Thus, the closure of $F_s$ is a subset of $S_s$.

    Next, observe that if $v \in T^\circ$, so that all the coordinates $v_i$ of $v$ are non-zero and $c$ is any member of $\scr C\backslash \{ p_v \}$,
    then by strict propriety $E_{p_v}(\hat s(v)) < E_{p_v}(\hat s(v'))$, so $E_{p_v}(\hat s(v))<\infty$. If $\hat s(v)=(a_1,...,a_n)$, then
    $E_{p_v}(\hat s(v)) = \sum_i v_i a_i$ since the $v_i$ are all non-zero, and so $a_i$ is finite for all $i$. Thus,
    $\hat s[T^\circ] \subseteq F_s$. Fix $u \in S_s\backslash F_s$. Then $u=\hat s(v)$ for some $v\in\partial T$. Then there will be a sequence
    in $T^\circ$ converging to $v$, and its image under $\hat s$ will be a sequence in $F_s$ converging to $\hat s(v)$ by continuity.
    Thus, the closure of $F_s$ is equals $S_s$, and the proof is complete by Theorem~\ref{th:domination}.
\end{proof}

To prove Theorem~\ref{th:domination}, we will transform the setting to a geometric one. Let $z = \hat s(c)$. Given two points $a$ and $b$
in $(-\infty,\infty]^n$, define the inner product
$$
    \< a, b \> = \sum_{i\in \{ i : 1\le i\le n, a_i\ne 0, b_i\ne 0 \}} a_i b_i.
$$
If both $a$ and $b$ are in $\R^n$ this is the usual inner product.
It is easy to see that $\< \cdot, v \>$ is a continuous function on $(-\infty,\infty]^n$.
A closed half-space in $\R^n$ with outward normal $v\in \R^n\backslash \{ 0 \}$ is
any set  of the form:
$$
    \{ z \in \R^n : \< z,v \> \le a \}
$$
for some $a\in\R$. An open half-space is one defined as above but with strict inequality. The interior $H^\circ$ of a closed half-space is the
corresponding open half-space. If $H$ is a half-space in $\R^n$ with outward normal $v$, say that its extension $H^*$ is the
set
$$
    \{ z \in (-\infty,\infty]^n : \exists w \in H(\<z, v\> \le \< z, w \>) \}.
$$

Say that a closed half-space $H$ separates a set $A$ from a set $B$ provided that $A\subseteq H$ and $B \subseteq \R^n\backslash H^\circ$.
The hyperplane separation theorem says that if $A$ and $B$ are convex and disjoint subsets of $\R^n$, there is such an $H$.

Let $C = (-\infty,0]^n$ be the non-positive orthant of $\R^n$. For a point $z \in (-\infty,\infty]^n$, let $z_i$ be its $i$th coordinate.
For any $v\ne 0$ and $D\subseteq \R^n$, let
$$
    \sigma_D(v) = \sup_{z\in D} \< z, v \>
$$
be the support function of $D$. 

\begin{lem}\label{lem:cone1}
Fix a point $z \in [0,\infty]^n$. Let $D$ be a subset of $\R^n$. Suppose that for all $v\in C$, the support function $\sigma_D(v)$ is finite.
Let $H_v = \{ w : \< w, v \> \le \sigma_D(v) \}$.
Suppose further that $z \in (H_v^\circ)^*$ for every non-zero $v\in C$. Then there is a $z'$ in the
convex hull $\Conv(D)$ such
that $z'_i < z_i$ for all $i$.
\end{lem}

\begin{lem}\label{lem:cone2}
Let $D$ be a subset closed in $\R^n$ such that for every non-zero $v\in C$, there is a unique closed
half-space $H_v$ with outward normal $v$ such that $D\subseteq H_v$ and the boundary of $H_v$
intersects $D$ at exactly one point. Then for any $z\in \Conv(D)$, there is a $z'\in D$ such
that $z'_i \le z_i$ for all $i$.
\end{lem}

\begin{proof}[Proof of Lemma~\ref{lem:cone1}]
Let $Q = \{ w \in \R^n : \forall i(w_i < z_i) \}$.
Suppose first that $Q$ does not intersect the convex hull $\Conv(D)$.
Then let $H$ be a half-space with outward normal $v$ that separates $\Conv(D)$ from $Q$.
It follows that there is an $a\in\R$ such that for every $w \in Q$, we have $\< w, v \> \ge a$.
Suppose $v_j > 0$ for some $j$. Choose any member $w'$ of $Q$. Let $b = \< w', v \>$. Let $w_i = w'_i$ for $i\ne j$ and
let $w_i = w'_i - (b-a+1)/v_j$. Then $w \in Q$ and $\< w, v \> = b - (b-a+1) = a-1 < a$, a contradiction. Thus, $v_j \le 0$
for all $j$, and hence $v \in C$.

Then $H_v$ and $H$ have the same outward normal, and since $H_v$ is the smallest closed half-space with that outward normal to contain $D$,
we must have $H_v \subseteq H$. It follows that $H_v$ separates $\Conv(D)$ from $Q$. Let $a=\sigma_D(v)$. Then $H_v = \{ z : \< z,v \> \le a \}$,
and for all $w\in Q$, we have $\< w,v \> \ge a$. Taking a sequence of members of $Q$ converging to $z$ and using the
continuity of $\< \cdot, v \>$, we conclude that $\< z,v \> \ge a$. But this contradicts the claim that $z \in (H_v^\circ)^*$.

Thus, $Q$ intersects $\Conv(D)$. Let $z'$ be any element of their intersection.
\end{proof}

\begin{proof}[Proof of Lemma~\ref{lem:cone2}]
Observe that $\Conv(D)\cap (z+C)$ is bounded. To see this, let
$v=(-1/n,...,-1/n)$. Then $\Conv(D) \subseteq H_v$. But it is easy to see that the $H_v \cap (w+C)$ must be bounded.

Thus, $\Conv(D)\cap (w+C)$ is a compact set. Let $z'$ be the point of that set furthest from $z$.
Any point in $z'+C$ other than $z'$ will be even further from $z$, so $\Conv(D)\cap (z'+C)$ is empty.

Let $H'$ be a closed half-space that separates $\Conv(D)$ from $z'+C$. Since $C$ is self-dual, $H'$ has outward normal $v'$ in $C$. Then $D$
is contained in $H_{v'}$. Then only way $z'$ can be a convex combination of members of $D$, then, is if it is a convex combination of members of
$D$ on the boundary of $H_{v'}$. But there is only one member of $D$ on the boundary of $H_{v'}$ for $v'\in C$, and hence $z'$ must be equal to
that one member. Thus, $z'\in D$, and the proof is complete.
\end{proof}

\begin{proof}[Proof of Theorem~\ref{th:domination}]
Let $\phi:(-\infty,\infty]^\Omega \to (-\infty,\infty]^n$ be the homeomorphism such that $\phi(f)=u$ where $u_i = f(\omega_i)$.
Observe that $E_{p_v}(f) = \< \phi(f), v \>$ for $v\in T$. 
Let $s^*(v) = \phi(\hat s(v))$.

Let $D=\phi[F_s]$. Since all scores are non-negative, if $u\in C$ and $v\in D$, we have $\< v,u \> \le 0$, and so $\sigma_D(u)\le 0<\infty$ for all $u\in C$. 
Fix any $c\in \scr C\backslash \scr P$.
Note that 
For each $s^*(v) \in D$ for $v \in T$, we can let $H_v = \{ w \in (-\infty,\infty]^n : \< w, -v \> \le \< s^*(v), -v \> \}$.
Strict propriety together with the closure assumption then ensures that if $z=\phi(s(c))$, then the conditions of Lemma~\ref{lem:cone1} are satisfied. 
Let $z'$ be as
in the conclusion of Lemma~\ref{lem:cone1}.

Then the conditions of Lemma~\ref{lem:cone2} yield a $z'' \in D$ such that $z''_i \le z'_i$ for
all $i$. Letting $p\in\scr P$ be such that $z''=\phi(s(p))$, our proof is complete.
\end{proof}


\begin{thebibliography}{99}
\bibitem{PettigrewNew}
R. Pettigrew. 2022. ``Accuracy-First Epistemology Without the Additivity Axiom'', \textit{Philosophy of Science}
    89:128--151.


\bibitem{Predd}
Joel B. Predd, Robert Seiringer, Elliott H. Lieb, Daniel N. Osherson, H. Vincent Poor, and Sanjeev R. Kulkarni. 2009.
``Probabilistic Coherence and Proper Scoring Rules'', \textit{IEEE Transactions on Information Theory} 55:4786--4792.

    \end{thebibliography}
\end{document}